\documentclass[10pt,reqno,a4paper]{amsart}

\usepackage{amsfonts,amssymb}

\title[Closed nodal surfaces]{Closed nodal surfaces for simply connected domains in higher dimensions}

\author{J. B.~Kennedy}

\dedicatory{\upshape
Group of Mathematical Physics, University of Lisbon\\
Av.~Prof.~Gama Pinto 2, 1649-003 Lisboa, Portugal\\[.5em]
\texttt{jkennedy@cii.fc.ul.pt}
}

\newtheorem{theorem}{Theorem}[section]
\newtheorem{lemma}[theorem]{Lemma}
\newtheorem{proposition}[theorem]{Proposition}
\newtheorem{corollary}[theorem]{Corollary}

\theoremstyle{remark}
\newtheorem{remark}[theorem]{Remark}
\newtheorem{definition}[theorem]{Definition}

\numberwithin{equation}{section}

\newcommand{\R}{\mathbb{R}}
\newcommand{\N}{\mathbb{N}}

\DeclareMathOperator{\dist}{dist}

\DeclareMathOperator{\capacity}{cap}
\DeclareMathOperator{\interior}{int}
\newcommand*{\subsubset}{\subset\joinrel\subset}

\begin{document}

\begin{abstract}
	We give an example of a domain in dimension $N \geq 3$, homeomorphic to a ball and  with analytic boundary, for which the 		second eigenvalue of the Dirichlet Laplacian has an eigenfunction with a closed nodal surface. The domain is constructed via a 	sequence of perturbations of the domain of S.~Fournais [J. Differential Equations \textbf{173} (2001), 145--159].
\end{abstract}



\maketitle

\section{Introduction}
\label{sec:intro}

For $\Omega \subset \R^N$, $N \geq 2$ a bounded domain, that is, a bounded, open, connected set, denote by $\{\lambda_j\}_{j=1}^\infty$ the eigenvalues of the Dirichlet Laplacian on $\Omega$, ordered by increasing size and repeated according to their multiplicites, and by $\{\psi_j\}_{j=1}^\infty$ a set of corresponding eigenfunctions, orthogonal in $L^2(\Omega)$. According to a well-known theorem of Courant, the nodal set $\mathcal{N} = \mathcal{N}(\Omega, j):= \overline{\{x\in\Omega:\psi_j(x)=0\}}$ can divide $\{x \in \Omega: \psi_j(x) \neq 0\}$ into at most $j$ distinct connected components, or nodal domains, of $\Omega$. As second eigenfunction $\psi_2$ must change sign in $\Omega$, it has exactly two nodal domains, $\Omega^+:=\{x\in\Omega: \psi_2(x)>0\}$ and $\Omega^-:= \{x \in \Omega: \psi_2(x)<0\}$. Despite considerable attention, for general domains $\Omega$ not much is known about $\Omega^\pm$, especially as regards their location. Of particular interest for over 40 years has been the  nodal line (or nodal domain) conjecture, which asserts that the nodal set of $\psi_2$ must intersect the boundary of $\Omega$:
\begin{equation}
	\label{eq:0}
	\mathcal{N} \cap \partial \Omega \neq \emptyset
\end{equation}
for any bounded domain $\Omega \subset \R^N$, meaning neither nodal domain can be compactly contained in $\Omega$. For $N=2$ this is usually attributed to Payne; see, e.g., \cite{payne:67}, Conjecture~5, p.~467 or \cite{payne:73}. The same question has been asked by others for higher dimensions; see, e.g., \cite{schoen:94}, Chapter~IX, Problem~45 (or \cite{fournais:01,freitas:08,kawohl:94}). Actually, we remark that Payne does not consider the conjecture to be his; cf.~his statement on p.~721 of \cite{payne:73} that it is of ``uncertain origin".

Various partial results have been established, showing that \eqref{eq:0} holds for certain classes of domains in $\R^2$, usually assuming some convexity and/or symmetry properties; see for example \cite{alessandrini:94,kawohl:94,liboff:05,lin:87,melas:92,payne:73} and the references therein. Perhaps the most general result so far is that \eqref{eq:0} holds for any convex domain in $\R^2$. Less is known in $\R^N$; as yet it remains quite a signifcant open problem to establish \eqref{eq:0} for general convex domains in $\R^N$. So far, this has been done in the special case of thin convex domains \cite{jerison:95}, and domains convex and symmetric with respect to at least two orthogonal directions \cite{damascelli:00}. Possibly the only positive result for non-convex domains in $\R^N$ is for thin curved tubes \cite{freitas:08}.

In the other direction, there is now a considerable body of counterexamples to the nodal domain conjecture in full generality. We will look at a property slightly stronger than the converse to \eqref{eq:0}: we are interested in whether
\begin{equation}
	\label{eq:1}
	{\{x\in\Omega:\psi_2(x) \leq 0\}} \subsubset \Omega,
\end{equation}
for some eigenfunction $\psi_2$ of $\Omega$, where $U \subsubset V$ denotes compact containment, that is, $U \subset K \subset V$ for some compact $K \subset \R^N$. It was shown in \cite{hoffmann:97} that there exists a multiply connected, bounded domain in $\R^2$ for which \eqref{eq:1}  holds. The result was subsequently extended in \cite{fournais:01} to $\R^N$, although using a different method of proof based on Brownian motion techniques rather than symmetry arguments. It has also been shown that \eqref{eq:1} can hold on a ball if one adds a suitable potential \cite{lin:88}, and on a simply connected surface with boundary \cite{freitas:02}. One can also find a simply connected, unbounded planar domain whose nodal set does not touch the boundary (but both nodal domains do) \cite{freitas:07}.

However, there is still a wide gulf between the two sides. One of the major open questions seems to be to determine the minimum possible connectivity amongst all domains $\Omega$ which satisfy \eqref{eq:1}; cf.~\cite{hoffmann:97}, Remark~3, or \cite{freitas:08}, Conjecture~1.2. This is perhaps most interesting in two dimensions, but has also been asked in $\R^N$; again, see \cite{freitas:08}, Conjecture~1.2, or also \cite{fournais:01}, where some effort was also devoted to making sure the constants involved were explicitly estimable. Here we give a complete answer in the higher dimensional case.

\begin{theorem}
\label{th:main}
	Let $N \geq 3$. There exists a bounded, open, non-convex set $\Omega \subset \R^N$ with analytic boundary, which is 		homeomorphic to the open unit ball in $\R^ N$, for which the second eigenvalue $\lambda_2 (\Omega)$ of the Dirichlet 			Laplacian is simple, and for which the corresponding eigenfunction $\psi_2$ satisfies \eqref{eq:1}.
\end{theorem}

This is the simplest possible topological structure in $\R^N$ (see Remark~\ref{rem:topology}), meaning topological considerations are eliminated completely when $N \geq 3$.  In fact it would be easy to modify our domain, using the same type of proof, to obtain a domain $\Omega \subset \R^N$, $N \geq 3$, with an essentially arbitrary topological structure, for which \eqref{eq:1} holds (see Remark~\ref{rem:4:others}). Theorem~\ref{th:main} also implies that regularity plays no r\^ole in determining whether \eqref{eq:1} holds; our method of proof allows us to make our domain analytic with a modest amount of extra work. Thus in higher dimensions \eqref{eq:1} may be regarded as a purely \emph{geometric} condition. However, we stress at the outset that our domain and method of proof cannot be extended to $N=2$.

\begin{remark}
\label{rem:topology}
We will refer to the following three topological notions for sets in $\R^N$: (i) being homeomorphic to an (open) ball in $\R^N$, (ii) being contractible and (iii) being simply connected. Following \cite{hatcher:02}, we understand (ii) to mean homotopic to a point and (iii) path-connected and with trivial fundamental group. The latter holds for a bounded domain $\Omega \subset \R^N$, $N \geq 3$ if it is homeomorphic (or homotopic) to a large open ball with a finite collection of points removed from its interior (see \cite{hatcher:02}, Section~1.2, Exercise~3; this fails for $N=1,2$), implying (iii) is strictly weaker than (ii). Moreover, (ii) is strictly weaker than (i), as follows from Theorem~2.26 of \cite{hatcher:02}, or, e.g., Section~1.2, Exercise~20 there. Here, however, all our contractible domains will satisfy a slightly stronger property, which we may take as our definition: for us $U \subset \R^N$ is contractible if it has finite $k$-dimensional measure ($1 \leq k \leq N$), and is homeomorphic to a ball (open, closed or including only part of its boundary) in $\R^k$. By this non-standard definition, for \emph{domains} $\Omega \subset \R^N$, (i) and (ii) are equivalent.
\end{remark}

The idea of the proof of Theorem~\ref{th:main} is to perturb the domain from \cite{fournais:01}, which we will call $\Omega_0$ and describe in Section~\ref{sec:prelim}. The principle is that under topologically unrestrictive conditions on a sequence of domains $\Omega_n$, we can expect the eigenfunctions as well as the eigenvalues associated with $\Omega_n$ to converge to those of $\Omega_0$, and in a sense strong enough to preserve \eqref{eq:1}. Put differently, the property \eqref{eq:1} is robust to perturbations of the domain away from its interior nodal domain. We make this precise and give some conditions from the literature on the $\Omega_n$ in Section~\ref{sec:perturb}. The proof of Theorem~\ref{th:main}, in Section~\ref{sec:proof}, is done by making four successive perturbations. We first perturb $\Omega_0$ to obtain a domain $\widetilde\Omega$ with the same complicated topology but Lipschitz boundary. We then perturb $\widetilde\Omega$ to obtain a simply connected domain $\widehat \Omega$, by (roughly speaking) partitioning $\R^N \setminus \widetilde \Omega$ into contractible pieces (contractible in the sense of Remark~\ref{rem:topology}) using thin ``sheets" added to $\widetilde \Omega$. This is the core of the proof, and also the most delicate part. We then drill thin tubes, or ``fireman's poles" into $\widehat \Omega$ to link the contractible connected components of $\R^N \setminus \widehat\Omega$ in its interior to the outside, obtaining a contractible $\Omega^*$, and then, finally, we approximate $\Omega^*$ from the outside by a sequence of analytic domains to produce our $\Omega$.

\smallskip

\noindent{\textbf{Acknowledgements.}} The author would like to thank Pedro Freitas for suggesting the problem, and for many helpful discussions and comments. This work was supported by grant PTDC/MAT/101007/2008 of the FCT, Portugal.

\section{Notation, background results and the domain of Fournais \cite{fournais:01}}
\label{sec:prelim}

Here we fix our notation and give some preliminary results on the Dirichlet Laplacian and the domain from \cite{fournais:01}. We start with the following basic notation. Depending on which is more convenient, we denote a point $x \in \R^N$ either using Cartesian coordinates  $x=(x_1,\ldots,x_N)$ or polar coordinates $0 \neq x=(r,\theta)$ with $r\in (0,\infty)$ and $\theta \in S_1:= \mathbb S^{N-1}= \{x\in \R^N:|x|=1\}$, the unit sphere in $\R^N$. In a slight abuse of notation we will variously write either $\theta \in S_1$ or $(1,\theta) \in S_1$. We will use $\sigma_k(U)$ to denote the $k$-dimensional (surface) measure of a set $U \subset \R^N$, $k=1, \ldots, N-1$. The complement of a set $U \subset \R^N$ in $\R^N$ is $U^c$, and we write $U \subsubset V$ to mean there exists $K \subset \R^N$ compact with $U \subset K \subset V$. We will denote by $B_r(x)$ a ball in $\R^N$ centred at $x$ and of radius $r$; if $x=0$ we will just write $B_r$. For $0<r<s$ we denote by $A_{r,s}$ the open annular or shell-like region
\begin{equation}
\label{eq:2:annulus}
A_{r,s}:= \{x\in \R^N: r<|x|<s\} = B_s \setminus \overline B_r.
\end{equation}

Given a bounded domain, that is, a bounded, open, connected set $\Omega \subset \R^N$, we understand the Dirichlet Laplacian on $\Omega$, which we will denote by $-\Delta^D_\Omega$, to be the operator on $L^2(\Omega)$ associated with the problem
\begin{equation}
	\label{eq:2:dirichlet}
	\begin{aligned}
	-\Delta u &=f & \quad\quad &\text{in $\Omega$}\\
	u &=0 & &\text{on $\partial \Omega$},
	\end{aligned}
\end{equation}
where $f \in L^2(\Omega)$, and the problem is understood in the usual weak sense. By standard theory, $-\Delta^D_\Omega$ is self-adjoint and has a sequence of eigenvalues $0<\lambda_1<\lambda_2\leq \ldots \to \infty$, where each eigenvalue is repeated according to its (finite) multiplicity. For each $j\geq 1$, we may write $\lambda_j=\lambda_j(\Omega)$ to emphasise the domain with which $\lambda_j$ is associated. The associated eigenfunctions $\{\psi_j\}_{j=1}^\infty$ after a suitable normalisation form an orthonomal basis for $L^2(\Omega)$. For each $j$, $\psi_j \in H^1_0(\Omega) \cap C^\infty(\Omega)$ is in fact analytic in $\Omega$ (see, e.g, \cite{dautray:88}, Section~V.4).

The first eigenvalue $\lambda_1$ is simple, and the associated eigenfunction $\psi_1$ may be chosen strictly positive in $\Omega$; by orthogonality, all the other eigenfunctions change sign in $\Omega$. As previously noted, Courant's nodal domain theorem (see \cite{courant:53}, Section~IV.6) asserts that for each $j$, $\mathcal{N}=\mathcal{N}(\Omega,j) = \overline{\{x \in \Omega: \psi_j(x)=0\}}$ divides $\{x \in \Omega: \psi_j(x) \neq 0\}$ into at most $j$ connected components. When $j=2$ this means $\Omega^+= \{x \in \Omega: \psi_2(x)>0\}$ and $\Omega^-= \{x \in \Omega: \psi_2(x)<0\}$ are connected subsets of $\Omega$. Moreover, it is not hard to show that $\partial \Omega^+ \cap \Omega = \partial \Omega^- \cap \Omega = \{x \in \Omega: \psi_2(x)=0\}$ (combine properties of analytic functions with the maximum principle to show that $\psi_2$ must take both signs in any neighbourhood of a point $x \in \mathcal{N}$).

Let us now describe the domain from \cite{fournais:01}, which we assume to be centred at the origin. We start by choosing $0<R_1<R$ such that $\lambda_1(B_{R_1}) < \lambda_1(A_{R_1, R}) < \lambda_2(B_{R_1})$. To simplify notation, without loss of generality we may assume $R_1=1$. Then the unit sphere $S_1$ is the common boundary between $B_1$ and $A_{1,R}$. For a suitable finite collection of points $z^i \in S_1$, $i=1,\ldots, M$ (we use superscripts to avoid confusion with Cartesian coordinates), there exists $\varepsilon_0>0$ depending on the $z^i$ such that for all $\varepsilon \in (0,\varepsilon_0)$, the domain
\begin{equation}
\label{eq:2:fournais}
	\Omega_{\{z^i\},\varepsilon}:=B_1 \cup A_{1,R} \cup \Bigl(\bigcup_{i=1}^M B_\varepsilon(z^i)\Bigr)
\end{equation}
has one of its nodal domains, say $\Omega_{ \{z^i\}, \varepsilon}^-$, satisying $\Omega_{ \{z^i\}, \varepsilon}^- \subsubset B_1$ (so \eqref{eq:1} holds).

We now choose the $z^i$ and a corresponding $\varepsilon(\{z^i\})>0$, which will be fixed throughout, although in the proof of Theorem~\ref{th:main} we will make a more precise choice of $\varepsilon$, one depending only on the $z^i$. We declare $\Omega_0 := \Omega_{ \{z^i\}, \varepsilon}$ to be our starting point, the domain to be perturbed. In words, $\Omega_0$ may be thought of as an inner ball $B_1$ and an outer thick shell $A_{1,R}$ joined by a large number of small sets $B_\varepsilon(z^i) \cap S_1$, which, considered as open sets in the manifold $S_1$,  we think of as circular, disconnected, $N-1$-dimensional ``rooms". We will set
\begin{equation}
	\label{eq:2:rooms}
	W_0:=\Bigl(\bigcup_{i=1}^M B_\varepsilon(z^i)\Bigr) \cap S_1 = \Omega_0 \cap S_1
\end{equation}
to be this collection of rooms; $S_1$ is partitioned into $W_0$ and its complement, which may be written as $\partial\Omega_0 \cap S_1$, and which is (multiply) connected.

\section{Some basic results on domain perturbation}
\label{sec:perturb}

The key idea of the proof is that we can approximate our given domain $\Omega_0$ by a sequence $\Omega_n$ and, under fairly broad conditions, expect convergence of the eigenfunctions in a reasonable sense. In this section, we will make this idea precise and introduce the general results on convergence we will use to prove Theorem~\ref{th:main}. We first observe that for any domain $\Omega \subset \R^N$, we may continuously embed $H^1_0(\Omega)$ into $H^1(\R^N)$ by extending functions $u \in H^1_0(\Omega)$ by $0$ outside $\Omega$. In this sense $H^1(\R^N)$ and $L^2(\R^N)$ are the natural spaces in which to consider sequences of functions defined on sequences of domains. So from now on we will regard elements of $H^1_0(\,.\,)$ as lying in $H^1(\R^N)$ (and similarly $L^2(\R^N)$) without further comment. 

\begin{definition}
	We will write $\Omega_n \to \Omega$ if the following two conditions are satisfied:
	\begin{itemize}
	\item[(i)] the weak limit points of every sequence $u_n \in H^1_0 (\Omega_n)$ in $H^1(\R^N)$ lie in $H^1_0(\Omega)$;
	\item[(ii)] for all $v \in H^1_0(\Omega)$ there exist $v_n \in H^1_0 (\Omega_n)$ with $v_n \to v$ in $H^1(\R^N)$.
	\end{itemize}
\end{definition}

In such a case it is often said that $\Omega_n \to \Omega$ in the sense of Mosco (as in \cite{daners:03}), since (i) and (ii) were probably first used in \cite{mosco:69}.

\begin{theorem}
\label{th:3:conv}
Suppose $\Omega_n,\Omega$ are bounded open sets in $\R^N$, such that $\Omega_n,\Omega \subset B_\rho$ for some fixed $\rho>0$ sufficiently large, and suppose that $\Omega_n \to \Omega$. If $\lambda_j$ is a simple eigenvalue of $\Delta^D_\Omega$, then the corresponding eigenvalue $\lambda_{j,n}$ of $\Delta^D_{\Omega_n}$ is simple for all $n \geq 1$ sufficiently large, $\lambda_{j,n} \to \lambda_j$ as $n \to \infty$, and after a suitable normalisation the corresponding eigenfunctions $\psi_{j,n} \to \psi_j$ in $L^2(\R^N)$.
\end{theorem}

\begin{proof}
Combine \cite{daners:03}, Corollary~4.7, which implies convergence of the corresponding resolvent operators, with Corollary~4.2, which then implies convergence of any finite part of the spectrum as well as the corresponding spectral projection in norm. 
\end{proof}

The abstract result on spectral convergence is really due to Kato (see Section~IV.3 of \cite{kato:76}), but needs the modifications of \cite{daners:03} to be applicable to this case.

We strengthen Theorem~\ref{th:3:conv} using standard estimates for solutions to elliptic equations to obtain local convergence in the supremum norm, not the strongest possible result but obviously sufficient for our purposes. Notationally, for the rest of the paper we will drop the subscript $j$ from our eigenvalues, functions and nodal sets, there being no danger of confusion as we are only interested in the case $j=2$.

\begin{theorem}
\label{th:3:cont}
Suppose $\Omega_n \to \Omega$ satisfy the conditions of Theorem~\ref{th:3:conv}, and that $\lambda_n \to \lambda$ are corresponding simple eigenvalues, with eigenfunctions $\psi_n \to \psi$ in $L^2(\R^N)$. Suppose also that there exists
\begin{displaymath}
	\Omega'\subsubset \Omega \cap \Bigl(\bigcap_{n \geq 1}\Omega_n \Bigr).
\end{displaymath}
Then $\psi_n \to \psi$ in $C(\overline{\Omega'})$.
\end{theorem}

\begin{proof}
By assumption, $d^*:=\min\{\inf_{n\geq 1}\dist(\Omega',\partial\Omega_n),\dist(\Omega',\partial\Omega)\}>0$. Fix $d \in (0,d^*)$. We will apply Theorem~8.24 of \cite{gilbarg:83} to the equation
\begin{displaymath}
	(\Delta+\lambda)(\psi-\psi_n)=(\lambda-\lambda_n)\psi_n
\end{displaymath}
on $\Omega'\subsubset \Omega'_d:=\{x \in \R^N: \dist(x,\Omega')<d\}$. Denote by $\Lambda \geq \sup\{\lambda, \lambda_n\}$  any fixed upper bound and also fix some $q>N/2$. Then there exist constants $C=C(N,d,\Lambda,q)>0$ and $\alpha=\alpha(N,\Lambda d)>0$ such that
\begin{equation}
	\label{eq:3:cont}
	\|\psi-\psi_n\|_{C^\alpha(\overline{\Omega'})}\leq C(\|\psi-\psi_n\|_{L^2(\Omega'_d)}
	+\|(\lambda-\lambda_n)\psi_n\|_{L^q(\Omega'_d)}).
\end{equation}
Since $\|\psi-\psi_n\|_{L^2(\Omega'_d)} \leq \|\psi-\psi_n\|_{L^2(\R^N)} \to 0$ and $\lambda_n \to \lambda$, we are done if we can show $\|\psi_n\|_{L^q(\Omega'_d)}$ remains bounded in $n$. We estimate it by the $L^2$-norm as follows:
\begin{displaymath}
\begin{split}
	\|\psi_n\|_{L^q(\Omega'_d)} \leq \|\psi_n\|_{L^q(\Omega_n)}&\leq C_1(N,|\Omega_n|)\|\psi_n\|_{L^\infty(\Omega_n)}\\
	&\leq C_1(N,|\Omega_n|) C_2(N,\Lambda,|\Omega_n|)\|\psi_n\|_{L^2(\Omega_n)},
\end{split}
\end{displaymath}
where the last inequality follows from Theorem~8.15 of \cite{gilbarg:83} applied to $(\Delta+\lambda_n)\psi_n=0$ on $\Omega_n$, $\psi_n = 0$ on $\partial \Omega_n$, and $\Lambda$ is as before. Since $\psi_n\to\psi$ in $L^2(\R^N)$, the $L^2$-norm of $\psi_n$ certainly stays bounded as $n\to\infty$. Moreover, since $|\Omega'| \leq |\Omega_n| \leq |B_\rho|$ for some large fixed $\rho>0$, $C_1$ and $C_2$ are bounded in $n$ and so the right hand side of \eqref{eq:3:cont} must indeed converge to 0 as $n\to\infty$.
\end{proof}

\begin{corollary}
\label{cor:3:main}
Suppose under the conditions and notation of Theorem~\ref{th:3:cont} that $\lambda$, $\lambda_n$ are the second eigenvalues of their respective domains, and that $\Omega^- = \{x\in \Omega:\psi(x)<0\}$ satisfies $\Omega^- \subsubset \Omega'$. Then for all $n \geq 1$ sufficiently large, the corresponding nodal domain $\Omega_n^- = \{x \in \Omega_n: \psi_n(x)<0\}$ also satisfies $\Omega_n^- \subsubset \Omega'$.
\end{corollary}

Finally for this section, we will state two relatively simple results which we will use to verify the condition $\Omega_n \to \Omega$. The first involves approximation of sets from the outside and the second from the inside. We will use $\capacity(E)$ to denote the $H^1$-capacity of a set $E$, given (as is standard) by
\begin{displaymath}
	\begin{split}
	\capacity(E):=\inf\{\|\varphi\|_{H^1}^2 &: \varphi \in H^1_0(\R^N) \text{ and $\varphi\geq 1$ in an}\\
	&\text{open neighbourhood of $E$} \}
	\end{split}
\end{displaymath}
(cf., e.g., Section~2.35 of \cite{heinonen:93} or Section~2 of \cite{rauch:75}).

\begin{proposition}
\label{prop:3:out}
Suppose $\Omega_n, \Omega \subset \R^N$ are open, $\Omega_n \supset \Omega_{n+1} \supset \Omega$ for all $n \geq 1$, and $\Omega = \interior \bigcap_{n\in\N}\Omega_n$. If $\Omega$ is Lipschitz, then $\Omega_n \to \Omega$.
\end{proposition}

\begin{proof}
This is Proposition~7.4 of \cite{daners:03}, noting that the stability condition there certainly holds if $\Omega$ is Lipschitz.
\end{proof}

\begin{proposition}
\label{prop:3:in}
Suppose $\Omega_n, \Omega \subset \R^N$ are open, $\Omega_n \subset \Omega_{n+1} \subset \Omega$ for all $n \geq 1$, and $\capacity(\Omega \setminus \bigcup_n \Omega_n) = 0$. Then $\Omega_n \to \Omega$.
\end{proposition}

\begin{proof}
This follows directly from Theorem~7.5 of \cite{daners:03}; conditions (2) and (3) there are trivial since $\Omega_n \subset \Omega$. To obtain (1) it suffices to show that $\capacity(\Omega \setminus \bigcup_n \Omega_n) = 0$ implies $\capacity (\Omega \setminus \Omega_n) \to 0$ as $n \to \infty$. But since the $\Omega_n$ are nested, this follows from the basic properties of ($H^1$-)capacity as a Choquet capacity; see for example \cite{heinonen:93}, Theorem~2.37 (cf.~also p.~32 and Theorem~2.2(iv) there). Alternatively, we could refer to Theorem~2.3 of \cite{rauch:75} directly for the convergence of the associated resolvent operators in this case.
\end{proof}

\section{Proof of Theorem~\ref{th:main}}
\label{sec:proof}

Corollary~\ref{cor:3:main} reduces the proof of Theorem~\ref{th:main} to finding a sequence of uniformly bounded, analytic domains $\Omega_n \to \Omega_0$, each homeomorphic to a ball, which leave invariant the set $B_1$ containing compactly one of the nodal domains of $\Omega_0$. Since under such conditions the simplicity of $\lambda_2(\Omega_n)$ is automatic for $n$ sufficiently large, we shall make no further reference to this, assuming implicitly that we are always sufficiently far into the tail of our sequence of domains to guarantee that this property holds. Moreover, all our domains will be contained in some fixed ball, say, $B_{R+1}$, so we will likewise make no further reference to the boundedness condition.

Since we make four separate perturbations, we will split the proof up into the four corresponding steps. The first is not strictly necessary, but serves to make the domains smoother, making subsequent perturbations far easier, and can be done with minimal effort. So we start by turning the ``rooms" $B_\varepsilon(z^i) \cap S_1$ into genuine passages of some positive length in the radial direction by making the outer shell $A_{1,R}$ thinner. %(Using different perturbation results, we could shift $A_{1,R}$ out %instead, preserving its volume or thickness.)
In polar coordinates $(r,\theta)\in\R^N$, for $n \geq 1$ we set
\begin{equation}
\label{eq:4:first}
	\begin{aligned}
	\Omega_n &:=B_1 \cup A_{1+\frac{1}{n},R} \cup
	\{(r,\theta):  r\in[1,1+\frac{1}{n}], (1,\theta)\in \Omega_0\cap S_1\} \\
	&=B_1 \cup A_{1+\frac{1}{n},R}\cup \Bigl([1,1+\frac{1}{n}]\times W_0\Bigr),
	\end{aligned}
\end{equation}
where $W_0$ is as in \eqref{eq:2:rooms} and we recall $A_{r,s}$ is given by \eqref{eq:2:annulus}. Then $\Omega_n \subset B_R$ is Lipschitz (in fact, piecewise-$C^\infty$) for each $n$.

\begin{lemma}
\label{lemma:4:pert1}
	There exists $n_0 \geq 1$ such that the domain $\Omega_n$ satisfies \eqref{eq:1} for all $n\geq n_0$, with one nodal 		domain $\Omega_n^- \subsubset B_1$ for all such $n$.
\end{lemma}

\begin{proof}
Noting that $\Omega^-\subsubset B_1$ if and only if $\Omega^-\subsubset B_{1-\delta}$ for $\delta=\delta(\Omega_0)>0$ sufficiently small, applying Corollary~\ref{cor:3:main} to $B_{1-\delta} \subsubset \Omega_0 \cap \bigl(\bigcap_n \Omega_n\bigr)$, it suffices to prove $\Omega_n \to \Omega_0$. For this we use Proposition~\ref{prop:3:in}, noting that the $\Omega_n \subset \Omega_0$ are nested and increasing. Moreover, $\Omega_0\setminus \Omega_n \subset A_{1,1+\frac{1}{n}}$, implying $\bigcup_n \Omega_n = \Omega_0$.
\end{proof}

We now pick $n_0 \geq 1$ as in Lemma~\ref{lemma:4:pert1} and declare $\widetilde \Omega:= \Omega_{n_0}$; this will be our new fixed domain to be perturbed. We will next construct a sequence of simply connected (but not contractible) domains $\widetilde \Omega_m \to \widetilde \Omega$. First we consider only what happens on the surface of the sphere $S_1$. The basic idea is to connect the rooms $W_0 = B_\varepsilon(z^i) \cap S_1$ with a finite number of thin passages lying on the surface $S_1$, of $N-1$-dimensional measure of order $m^{-1}$. These passages, whose number and location will be determined purely by the $z^i$, and thus be independent of $m$ as well as the $n$ from the previous perturbation, will make $\widetilde \Omega_m \cap S_1$ (multiply) connected and partition $S_1 \setminus \widetilde \Omega_m$ into a fixed number $J\in\N$ of disjoint pieces, contractible in the sense of Remark~\ref{rem:topology}.

To do this we construct a closed set $G \subset S_1$, such that (i) $\sigma_{N-2}(G)<\infty$, (ii) $z^i \in G$ for all $i$, (iii) $G$ is connected, and (iv) $S_1 \setminus G$ is the finite, disjoint union of $J$ contractible connected components. Any set $G \subset S_1$ with these properties will do; we construct one explicitly as follows. Writing $z^i = (z_1^i, \ldots, z_N^i)$ in Cartesian coordinates for each $i$, we set
\begin{equation}
	\label{eq:4:g}
	G=\Bigl(\{x\in\R^N:x_1=0\} \cup \Bigl(\bigcup_{i=1}^M\{x\in\R^N:x_N=z_N^i\}\Bigr)\Bigr) \cap S_1.
\end{equation}
That is, we intersect $S_1$ with a vertical hyperplane $\{x_1=0\}$ and a collection of at most $M$ horizontal hyperplanes $\{x_N=z_N^i\}$, one through each of the $z^i$. It is easy to see $G$ has the desired properties (i)--(iv); $S_1 \setminus G$ is divided into some number $J \leq 2M+2$ connected components, the precise number depending on the number of points $z^i$ whose $x_N$-values coincide. We denote the $J$ components by $V_k$, $k = 1, \ldots, J$; each of the $V_k$ will look like
\begin{equation}
	\label{eq:4:vk}
	V_k = \{x \in S_1: x_1 < 0 \text{ (or $x_1>0$) and } x_N \in (z_N^i, z_N^j)\}
\end{equation}
for some appropriate $i$ and $j$.

In order to construct a smooth domain, we now specify the choice of $\varepsilon>0$ in $\Omega_{ \{z^i\}, \varepsilon}=\Omega_0$ more carefully, noting that this will depend only on the $z^i$, and is permissible since the result of \cite{fournais:01} still holds if $\varepsilon>0$ is made smaller. We first specify that $B_\varepsilon(z^i) \cap B_\varepsilon(z^j) = \emptyset$ for all $i \neq j$. We also declare that
\begin{displaymath}
	\varepsilon \in (0, \frac{1}{4}\min\bigl\{ \min\{|z_N^i-z_N^j|: i\neq j\}, \min \{|z_1^i|: z_1^i \neq 0\}\bigr\}),
\end{displaymath}
this interval being non-empty. This means that for each $i$, the only horizontal plane intersecting $B_\varepsilon(z^i)$ is $\{x_N = z_N^i\}$, while $B_\varepsilon(z^i)$ intersects $\{x_1=0\}$ if and only if $z_1^i = 0$. Moreover, the intersections are orthogonal.

We fill out $G$ together with the balls $B_\varepsilon(z^i)$ to make the surface on $S_1$ to be used in our perturbing domains. For each $m \geq 1$, we set
\begin{displaymath}
	G_m:=\{x \in S_1:\dist(x,G)<\frac{\varepsilon}{m}\},
\end{displaymath}
where $\dist(\,.\,,\,.\,)$ is the usual Euclidean distance in $\R^N$. Consider now the nested sequence $G_m \cup W_0$, with $W_0$ from \eqref{eq:2:rooms} (and see \eqref{eq:4:first}). Each $G_m \cup W_0$ contains $\widetilde \Omega\cap S_1 = \Omega_0 \cap S_1 = W_0$, each is connected, they are nested with $\bigcap_m (G_m \cup W_0) = G \cup W_0$, $\sigma_{N-1}(G_m \cup W_0) \leq \infty$ and $\sigma_{N-1}(G_m \setminus \widetilde\Omega) \to 0$. Moreover, by choice of our $\varepsilon>0$, each $G_m \cup W_0$ has boundary inside $S_1$ which is locally the graph of a Lipschitz (piecewise-$C^\infty$) function, since by choice of $\varepsilon>0$ the edges of the $G_m$ considered as open subsets of the manifold $S_1$ intersect the smooth edges of the balls $B_\varepsilon(z^i) \cap S_1$ transversally. Finally, $S_1 \setminus (G_m \cup W_0)$ is still the disjoint union of $J$ contractible connected components, which we will call $V_{1,m},\ldots,V_{J,m}$. For each $k$, the $V_{k,m}$ are nested, with $\bigcup_m V_{k,m} = V_k$ (given by \eqref{eq:4:vk}).

We can now describe the approximating domains $\widetilde \Omega_m$. To do so more conveniently, switching back to polar coordinates $x=(r,\theta) \in \R^N$ and writing $\theta \in G_m \subset S_1$ interchangeably with $(1,\theta) \in G_m$, we set
\begin{displaymath}
	\begin{split}
	\widetilde \Omega_m &:= \widetilde\Omega\cup\{(r,\theta)\in\R^N: r\in [1,1+\frac{1}{n_0}],\theta\in G_m\}\\
	&= B_1 \cup A_{1+\frac{1}{n_0},R} \cup \Bigl([1,1+\frac{1}{n_0}] \times (G_m \cup W_0)\Bigr),
	\end{split}
\end{displaymath}
where $n_0$ is the integer such that $\widetilde\Omega=\Omega_{n_0}$. In words, to make $\widetilde\Omega_m$ we start with $\widetilde\Omega$, and add to it a copy of $G_m$ at every radial level $r \in [1,1+\frac{1}{n_0}]$.

In particular, $\widetilde \Omega_m$ is Lipschitz (and piecewise-$C^\infty$) since $\widetilde\Omega$ and $[1,1+\frac{1}{n_0}] \times G_m$ are, and they intersect each other non-tangentially; and $\R^N \setminus \widetilde \Omega_m$ consists of the unbounded outer set $\{x\in \R^N:|x| \geq R\}$ together with $J$ contractible sets (``holes" in $\widetilde \Omega_m$) of the form
\begin{equation}
	\label{eq:4:holes}
	U_{k,m} = [1,1+\frac{1}{n_0}] \times V_{k,m},
\end{equation}
$k=1,\ldots,J$ (again in polar coordinates) , where $V_{k,m}$ is one of the contractible pieces of $\partial \widetilde \Omega_m \cap S_1$ as above. Topologically, this means $\widetilde \Omega_m$ is homeomorphic to a large ball with $J$ disjoint closed small balls removed from its interior. As $N \geq 3$, $\widetilde \Omega_m$ is simply connected (see Remark~\ref{rem:topology}), although we will only use this indirectly.

\begin{lemma}
\label{lemma:4:pert2}
	There exists $m_0 \geq 1$ such that $\widetilde\Omega_m$ satisfies \eqref{eq:1} with one nodal domain ${\widetilde 			\Omega}_m^- \subsubset B_1$, for all $m \geq m_0$.
\end{lemma}

\begin{proof}
Using the same reasoning as in Lemma~\ref{lemma:4:pert1}, it suffices to show that $\widetilde \Omega_m \to \widetilde \Omega$. This time we shall use Proposition~\ref{prop:3:out}, noting that $\widetilde \Omega \subset \widetilde \Omega_m \subset \widetilde \Omega_{m-1}$. As $\widetilde \Omega$ is Lipschitz, we only have to verify that $\widetilde \Omega = \interior(\bigcap_{m \geq 1}\widetilde \Omega_m)$. Since $\bigcap_{m \geq 1} G_m = G$, $\bigcap_{m \geq 1}\widetilde \Omega_m$ will consist of $\widetilde\Omega$ together with a thin set $T$ (a collection of ``sheets", to mix our metaphors), consisting of a copy of $G$ at each radial level, that is,
\begin{displaymath}
	\bigcap_{m\geq 1}\widetilde\Omega_m = \widetilde\Omega \cup T
	:=\widetilde\Omega \cup \Bigl([1,1+\frac{1}{n_0}] \times G\Bigr).
\end{displaymath}
Then $\sigma_{N-1}(T)<\infty$ as $\sigma_{N-2}(G)<\infty$. Moreover, $T$ intersects $\partial\widetilde\Omega$ only on a set of finite $N-2$-dimensional measure, which we may consider as consisting of three parts: the set $G \setminus W_0$ in $S_1$, a copy of this in $S_{1+\frac{1}{n_0}}=  \{|x|=1+\frac{1}{n_0}\}$, and the set where $T$ intersects the boundary of the passages $[1,1+\frac{1}{n_0}] \times W_0$ (cf.~\eqref{eq:2:rooms} and \eqref{eq:4:first}). In particular, as $\widetilde\Omega$ is Lipschitz, $\interior(\widetilde \Omega \cup T) = \widetilde \Omega$, and we may apply Proposition~\ref{prop:3:out} to obtain $\widetilde\Omega_m\to\widetilde\Omega$.
\end{proof}

We now fix $m_0 \geq 1$ as in Lemma~\ref{lemma:4:pert2} and take $\widehat \Omega:=\widetilde \Omega_{m_0}$ as our new fixed domain. We can now construct domains $\widehat \Omega_l \to \widehat \Omega$ homeomorphic to a ball. For each of the $J$ holes $U_k:= U_{k,m_0}$ given by \eqref{eq:4:holes}, we pick a point $(1,\theta_k) \in V_{k,m_0} = \partial\widehat\Omega \cap S_1$, such that there exists $\eta>0$ such that $\dist((1,\theta),\partial V_{k,m_0})\geq\eta$ for all $k$ (that is, $\theta_k$ is away from the edge of $V_{k,m_0}$, considered as a set in the manifold $S_1$). Since there are finitely many $V_{k,m_0}$, each with non-empty interior, we can certainly do this. Also note that   $(r,\theta_k) \in U_k$ for all $r \in [1, 1+\frac{1}{n_0}]$.  For each $k$ we consider the one-dimensional line
\begin{displaymath}
	L_k:=\{(r,\theta_k) \in \R^N: r\in (1+\frac{1}{n_0},R)\} \subset \widehat\Omega
\end{displaymath}
and we construct $\widehat \Omega_l$ by removing $J$ ``fireman's poles" (cf.~the introduction to \cite{rauch:75}) of width $l^{-1}$ centred about the $L_k$ from $\widehat\Omega$:
\begin{displaymath}
	\widehat\Omega_l:=\widehat\Omega\setminus\{x\in\widehat\Omega:\dist(x,L_k)\leq\frac{1}{l},
	\text{ for some $k=1,\ldots,j$}\}.
\end{displaymath}
$l \geq 1$. We assume without loss of generality that $l \geq \eta^{-1}$, so that these fireman's poles are thin compared with the holes $U_k$, and meet the $U_k$ approximately orthogonally. Since in addition these holes are cylindrical, the $\widehat \Omega_l$ are still Lipschitz and piecewise-$C^\infty$. Moreover, the $\widehat \Omega_l$ are now contractible, and indeed homeomorphic to balls in $\R^N$; the same is true of $B_{R+1} \setminus \widehat\Omega_l$, and $\R^N \setminus \widehat\Omega_l$ is simply connected.

\begin{lemma}
\label{lemma:4:pert3}
	There exists $l_0 \geq 1$ such that $\widehat\Omega_l$ satisfies \eqref{eq:1} with one nodal domain $\widehat 			\Omega_l^- \subsubset B_1$, for all $l \geq l_0$.
\end{lemma}

\begin{proof}
As usual, we only have to show $\widehat\Omega_l \to \widehat \Omega$, which we do using Proposition~\ref{prop:3:in}. In this case, it is immediate that the $\widehat \Omega_l$ are nested subsets of $\widehat \Omega$, and $\widehat \Omega \setminus \bigcup_l \widehat\Omega_l$ is just the union of the line segments $L_k$, which as a one-dimensional object has zero capacity in $\R^N$ for $N \geq 3$ (see, e.g., \cite{rauch:75}, Section~2, Example~3).
\end{proof}

We now choose such an $l_0 \geq 1$ and set $\Omega^*:=\widehat\Omega_{l_0}$. Then $\Omega^*$ has every property claimed in Theorem~\ref{th:main}, except that its boundary is only Lipschitz and piecewise-$C^\infty$ rather than analytic. To construct our $\Omega$ from a final sequence of approximations, we first note that there exists $\delta_0>0$ such that the open set
\begin{displaymath}
	\Omega_\delta^* := \{x \in \R^N: \dist(x,\Omega^*)<\delta\}
\end{displaymath}
is still homeomorphic to a ball for all $\delta \in (0,\delta_0)$. Choose any sequence $\delta_p \to 0$ monotonically, with $\delta_1<\delta_0$. We shall find inductively a nested sequence of analytic domains $\Omega_p^*$, with $\Omega \subset \Omega_p^* \subset \Omega_{p-1}^* \cap \Omega_{\delta_p}^*$.

We start by using \cite{edmunds:87}, Theorem~V.4.20 to obtain the existence of a set $V_1$ with analytic boundary and $\Omega^* \subsubset V_1 \subsubset \Omega_{\delta_1}^*$. In case $V_1$ is not connected and/or not homeomorphic to a ball, we first choose the unique connected component $V_1' \subset V_1$ such that $\Omega^* \subsubset V_1'$, and then obtain $\Omega_1^*$ by ``filling in" any holes in $V_1'$. Noting that since $V_1'$ is bounded analytic, it will be homeomorphic to a large ball with a finite number, possibly zero, of connected sets removed from its interior, we let $\Omega_1^*$ be the unique contractible set such that $V_1' \subset \Omega_1^*$ and $\partial \Omega_1^* \subset \partial V_1'$. Thus $\Omega_1^*$ is an analytic domain homeomorphic to a ball with $\Omega^* \subset \Omega_1^* \subset \Omega_{\delta_1}^*$.

Proceeding inductively, given $\Omega_{p-1}^*$, we obtain via \cite{edmunds:87} an analytic set $V_p$ such that $\Omega^* \subsubset V_p \subsubset \Omega_{p-1}^* \cap \Omega_{\delta_p}^*$. We then repeat the process described above to extract an analytic $\Omega_p^* \subset V_p$ homeomorphic to a ball. 

\begin{lemma}
\label{lemma:4:pert4}
	There exists $p_0 \geq 1$ such that $\Omega_p^*$ satisfies \eqref{eq:1} with  one nodal domain $(\Omega_p^*)^- 			\subsubset B_1 \subsubset \Omega_p^*$ for $p \geq p_0$.
\end{lemma}

\begin{proof}
Again, we show $\Omega_p^* \to \Omega^*$. By construction, $\Omega_p^* \supset \Omega_{p+1}^* \supset \Omega^*$, and since $\Omega_p^* \subset \{x \in \R^N: \dist(x,\Omega^*) < \delta_p\}$, $\delta_p \to 0$, we have $\bigcap_{p\geq 1}\Omega_p^* = \overline{\Omega^*}$. As $\Omega^*$ is Lipschitz, Proposition~\ref{prop:3:out} implies $\Omega_p^* \to \Omega^*$.
\end{proof}

Choosing $\Omega$ to be as in Lemma~\ref{lemma:4:pert4} completes the proof of Theorem~\ref{th:main}.

\begin{remark}
\label{rem:4:others}
It is possible to construct domains with various other properties for which \eqref{eq:1} holds. For example, we could prove that, given any type of connectivity condition, we could find a smooth domain satisfying \eqref{eq:1} and that condition (e.g., doubly or multiply connected, or generally having any fundamental group), by modifying our $\Omega$ on a sufficiently small set away from $B_1$. We could also construct various unbounded domains satisfying \eqref{eq:1} using the same ideas but different perturbation results. We do not go into details.
\end{remark}

\bibliographystyle{amsplain}

\providecommand{\bysame}{\leavevmode\hbox to3em{\hrulefill}\thinspace}
%\providecommand{\MR}{\relax\ifhmode\unskip\space\fi MR }
%% \MRhref is called by the amsart/book/proc definition of \MR.
%\providecommand{\MRhref}[2]{%
%  \href{http://www.ams.org/mathscinet-getitem?mr=#1}{#2}
%}
\providecommand{\href}[2]{#2}

\end{document}